\documentclass[12pt]{amsart}

\usepackage{psfrag}
\usepackage{color}

\usepackage{graphicx,graphics}
\usepackage{fullpage,amssymb,amsfonts,amsmath,amstext,amsthm,amscd,enumerate,verbatim,tikz}
\usepackage[T1]{fontenc}
\usetikzlibrary{matrix,arrows}

\usepackage[all]{xy} \usepackage{url}

\begin{document}
\newcommand{\note}[1]{\marginpar{\tiny #1}}
\newtheorem{theorem}{Theorem}[section]
\newtheorem{result}[theorem]{Result}
\newtheorem{fact}[theorem]{Fact}
\newtheorem{conjecture}[theorem]{Conjecture}
\newtheorem{lemma}[theorem]{Lemma}
\newtheorem{proposition}[theorem]{Proposition}
\newtheorem{corollary}[theorem]{Corollary}
\newtheorem{facts}[theorem]{Facts}
\newtheorem{question}[theorem]{Question}
\newtheorem{props}[theorem]{Properties}
\theoremstyle{definition}
\newtheorem{example}[theorem]{Example}
\newtheorem{definition}[theorem]{Definition}
\newtheorem{remark}[theorem]{Remark}

\newcommand{\notes} {\noindent \textbf{Notes.  }}
\renewcommand{\note} {\noindent \textbf{Note.  }}
\newcommand{\defn} {\noindent \textbf{Definition.  }}
\newcommand{\defns} {\noindent \textbf{Definitions.  }}
\newcommand{\x}{{\bf x}}
\newcommand{\z}{{\bf z}}
\newcommand{\B}{{\bf b}}
\newcommand{\V}{{\bf v}}
\newcommand{\T}{\mathcal{T}}
\newcommand{\Z}{\mathbb{Z}}
\newcommand{\Hp}{\mathbb{H}}
\newcommand{\D}{\mathbb{D}}
\newcommand{\R}{\mathbb{R}}
\newcommand{\N}{\mathbb{N}}
\renewcommand{\B}{\mathbb{B}}
\newcommand{\C}{\mathbb{C}}
\newcommand{\dt}{{\mathrm{det }\;}}
 \newcommand{\adj}{{\mathrm{adj}\;}}
 \newcommand{\0}{{\bf O}}
 \newcommand{\w}{\omega}
 \newcommand{\av}{\arrowvert}
 \newcommand{\zbar}{\overline{z}}
 \newcommand{\htt}{\widetilde{h}}
\newcommand{\ty}{\mathcal{T}}
\renewcommand\Re{\operatorname{Re}}
\renewcommand\Im{\operatorname{Im}}
\newcommand{\diam}{\operatorname{diam}}
\newcommand{\dist}{\text{dist}}
\newcommand{\ds}{\displaystyle}
\numberwithin{equation}{section}
\newcommand{\cN}{\mathcal{N}}
\renewcommand{\theenumi}{(\roman{enumi})}
\renewcommand{\labelenumi}{\theenumi}
\newcommand{\inte}{\operatorname{int}}

\date{\today}
\title{On Julia limiting directions in higher dimensions}
\author{A. N. Fletcher}

\maketitle

\begin{abstract}
In this paper we study, for the first time, Julia limiting directions of quasiregular mappings in $\R^n$ of transcendental-type. First, we give conditions under which every direction is a Julia limiting direction. Along the way, our methods show that if a quasi-Fatou component contains a sectorial domain, then there is a polynomial bound on the growth in the sector. Second, we give a contribution to the inverse problem in $\R^3$ of determining which compact subsets of $S^2$ can give rise to Julia limiting directions. The methods here will require showing that certain sectorial domains in $\R^3$ are ambient quasiballs, which is a contribution to the notoriously hard problem of determining which domains are the image of the unit ball $\B^3$ under an ambient quasiconformal map of $\R^3$ to itself.
\end{abstract}

\section{Introduction}

Let $f:\C \to \C$ be a transcendental entire function. We say that $e^{i\theta} $ in the unit circle $S^1$ is a {\it Julia limiting direction} of $f$ if there exists a sequence $(z_n)_{n=1}^{\infty}$ contained in the Julia set $J(f)$ with $|z_n| \to \infty$ and 
\[ e^{i \arg z_n} \to e^{i\theta} \]
in $S^1$ as $n\to \infty$. We write $L(f)$ for the set of Julia limiting directions of $f$. This definition was introduced by Qiao in \cite{Qiao}, where it was shown that if the lower order of $f$ is a finite value $\lambda$, then $L(f)$ contains an interval in $S^1$ of Lebesgue measure at least $\min \{ 2\pi , \pi / \lambda \}$. In particular, if $\lambda \leq 1/2$, then every element of $S^1$ is a Julia limiting direction. Subsequently, $L(f)$ has been studied for transcendental meromorphic functions in \cite{QW,WY,Wang,ZWH}, as well as for solutions of difference and differential equations, see for example \cite{WC} and the references contained therein. We also observe that some authors call the set of Julia limiting directions the radial distribution of $J(f)$.

The inverse problem for Julia limiting directions is to determine which compact sets in $S^1$ can be realized as $L(f)$. For meromorphic functions, this question was settled by Wang and Yao \cite{WY}. For entire functions, this question is only interesting when the lower order is larger than $1/2$. There exists an entire function of infinite order with only one Julia limiting direction, see the paper \cite{Baker} of Baker. In general, the inverse problem was solved completely for functions of infinite lower order (and partially for finite lower order) in \cite{WY}.

In the present paper we study Julia limiting directions for quasiregular mappings $f:\R^n \to \R^n$ of transcendental-type, where $n\geq 2$. Quasiregular mappings provide the natural setting for the generalization of complex dynamics into higher (real) dimensions and there is a well developed theory of quasiregular dynamics. In particular, the Julia set of a quasiregular mapping of transcendental-type was defined by Bergweiler and Nicks in \cite{BN}. We may therefore make the following definition.

\begin{definition}
\label{def:1}
Let $n\geq 2$ and let $f:\R^n \to \R^n$ be a quasiregular mapping of transendental-type. 
\begin{enumerate}[(i)]
\item We say that $\xi \in S^{n-1}$ is a {\it Julia limiting direction} of $f$ if there exists a sequence $(x_n)_{n=1}^{\infty}$ contained in $J(f)$ with $|x_n| \to \infty$ and $\sigma(x_n) \to \xi$. Here, $\sigma(x) \in S^{n-1}$ is defined by $\sigma(x) = x/|x|$.
\item We denote by $L(f) \subset S^{n-1}$ the set of Julia limiting directions of $f$.
\end{enumerate}
\end{definition}

We point out here that for a quasiregular mapping $f:\R^n \to \R^n$ of polynomial-type, which can be characterized as a map of finite degree, whose inner dilatation is strictly smaller than the degree, then by \cite{Berg} and \cite{FN} the Julia set is a compact subset of $\R^n$ and hence there are no Julia limiting directions.

\section{Statement of results}

{\bf Notation:} Throughout, $\R^n$ denotes Euclidean $n$-space and $\R^n \cup \{ \infty \}$ denotes the one point compactification of $\R^n$. We may sometimes identify $\R^n \cup \{ \infty \}$ with $S^n$ in the usual way. If the context is clear, $S(x,r)$ denotes the sphere in $\R^n$ centred at $x$ of radius $r>0$. Then $S^{n-1}$ will denote the unit $(n-1)$-sphere $S(0,1)$ in $\R^n$. If $U$ is a proper sub-domain of $\R^n$, then $d(x,\partial U)$ denotes the Euclidean distance from $x$ to the boundary of $U$.

We collect some basic results about $L(f)$ in the following.

\begin{proposition}
\label{prop:1}
Let $n\geq 2$ and $f:\R^n\to \R^n$ be quasiregular of transcendental-type.
Then the set $L(f)$ is a closed, non-empty subset of $S^{n-1}$.
\end{proposition}

While we defer definitions until the next section, we can quickly prove this proposition here.

\begin{proof}
It is clear that $L(f)$ is closed. That $L(f)$ is non-empty relies on the fact that for a transcendental-type mapping, $J(f)$ is unbounded. To see why this is true, suppose not. Then there exists a neighbourhood of infinity $\{x:|x| >R \}$ whose image, by complete invariance of $J(f)$, must omit all of $J(f)$. By \cite[Theorem 1.1]{BN}, $J(f)$ is an infinite set, so by \cite[Theorem IV.2.27]{Rickman} $f$ has a limit at infinity. This is a contradiction.
\end{proof}

For our first main result, we generalize part of \cite[Theorem 1]{Qiao} to higher dimensions. Roughly speaking, this result says that if $f$ grows slowly, then every direction is a Julia limiting direction. For an entire quasiregular map, we denote by $M(r,f)$ and $m(r,f)$ the maximum and minimum modulus of $f$ on the sphere $S(0,r)$, that is,
\[ M(r,f) = \max_{|x| = r} |f(x)|, \quad m(r,f) = \min_{|x| =r} |f(x)|.\]

\begin{theorem}
\label{thm:1}
Let $n\geq 2$ and suppose $f:\R^n \to \R^n$ is a quasiregular mapping of transendental-type. Suppose further that there exist $\alpha >1$ and $\delta >0$ such that for all large $r$ there exists $s\in[r,\alpha r]$ such that $m(s,f) \geq \delta M(r,f)$. Then $L(f) = S^{n-1}$.
\end{theorem}

We remark that the hypothesis on $f$ here was used in \cite[Theorem 1.6]{BDF} to obtain the conclusion that the fast escaping set $A(f)$ is a spider's web. Moreover, this condition can be fulfilled by, for example, the mappings constructed by Drasin and Sastry in \cite{DS}. The reason for needing this condition is that currently the strongest conclusion of the quasiregular counterpart of Wiman's Theorem, see \cite[Corollary 5.10]{MMV}, is that $\limsup_{r\to \infty} m(r,f) = \infty$, which is too weak to be useful for our methods.
However, the proof of Theorem \ref{thm:1} will show that we could replace the condition on the minimum modulus by one of the form
\begin{equation}
\label{eq:mrf} 
\limsup_{r\to \infty} \frac{m(r,f)}{r^m} = \infty
\end{equation}
for any positive integer $m$. This condition will be enough to conclude that the quasi-Fatou set $QF(f)$ cannot contain any sectorial regions. We refer to \cite{Zheng} for more discussion on such conditions in complex dynamics, but just observe that a strengthening of Wiman-type results for quasiregular mappings would yield an improved version of Theorem \ref{thm:1}. The fact that $\log^+ |f(z)|$ is subharmonic for entire functions $f:\C \to \C$ has led to strengthening of Wiman's Theorem in the plane, see for example Hayman's book \cite{Hayman}. One could hope for analogous results in non-linear potential theory based on the properties of $\log^+ |f(x)|$ when $f$ is entire and quasiregular in $\R^n$ could yield conclusions such as \eqref{eq:mrf}.

For our second main result, we give a partial answer to the inverse problem in $\R^3$. In complex dynamics, since $L(f)$ is contained in $S^1$, every component is either an interval or a point. In higher dimensions, components of closed subsets of $S^2$ can be much more complicated.

To state our result, we need to recall some terminology, see for example \cite{Vaisala}, although we state our definitions using the chordal metric on $S^2$ instead of the Euclidean metric on $\R^2$. Let $D \subset S^2$ be a domain, and let $x\in \partial D$. The domain $D$ is called {\it finitely connected} at $x$ if for all sufficiently small neighbourhoods $U$ of $x$, $U\cap D$ has a finite number of components. A simply connected domain $D \subset S^2$ with locally rectifiable boundary is an {\it inner chordarc domain} if $D$ is finitely connected on the boundary and there exists $c>0$ such that for each pair of boundary points $x_1,x_2$ we have
\[ \min \{ \ell (\gamma_1) , \ell(\gamma_2) \} \leq c \inf \ell ( \beta ),\]
where $\ell$ denotes the chordal length of an arc, $\gamma_1,\gamma_2$ are the components of $\partial D \setminus \{x_1,x_2\}$ and the infimum is taken over all open arcs $\beta$ in $D$ which join $x_1$ and $x_2$.

To see that our use of $S^2$ does not introduce extra complications in the definition of an inner chordarc domain, if $D \subset S^2$ is an inner chordarc domain and $\infty \in \overline{D}$, then we can apply a chordal isometry $A$ so that $\infty \notin \overline{A(D)}$. Then since the chordal and Euclidean metrics are bi-Lipschitz equivalent on compact subsets of $\R^2$, it follows that $A(D)$ is an inner chordarc domain using the usual definition involving the Euclidean metric in $\R^2$.

\begin{definition}
\label{def:dica}
A Jordan curve $J \subset S^2$ is called a {\it double inner chordarc} curve if the two components $D_1,D_2$ of $S^2 \setminus J$ are inner chordarc domains.
\end{definition}

We can now state our theorem on the inverse problem.

\begin{theorem}
\label{thm:2}
For $m\in \N$, let $J_1,\ldots, J_m$ be double inner chordarc curves in $S^2$ so that $D_1,\ldots, D_m$ is a collection of domains which are, respectively, components of $S^2 \setminus J_j$, for $j=1,\ldots, m$ and, moreover, such that $\overline{D_j}$ are pairwise disjoint.
Set $E = \cup_{j=1}^n \overline{D_j}$. Then there exists a quasiregular map $f_E:\R^3 \to \R^3$ of finite order with $L(f_E) = E$.
\end{theorem}

By considering the second iterate $f^2$ instead of $f$, noting that the constructed map $f_E$ is locally Lipschitz and applying \cite[Theorem 1.6 (v)]{BN}, one can also produce an example of infinite order satisfying the conclusion of Theorem \ref{thm:2}. We restrict to $\R^3$ because our proof relies on a construction of Nicks and Sixsmith in \cite{NS18} which yields a quasiregular map of transcendental-type that is the identity in a half-space. The double inner chordarc condition is required for our approach because it guarantees the sector in $\R^3$ which in spherical coordinates has cross-section $D_j$ is an ambient quasiball, in other words, there is an ambient quasiconformal map $g:\R^3 \to \R^3$ with $g(D_j)$ equal to a half-space via results of Gehring \cite{Gehring} and V\"ais\"al\"a \cite{Vaisala}. The connection with the construction of Nicks and Sixsmith is then apparent.

The paper is organized as follows. In section 3, we recall material from the various topics we will need to prove our main results. In section 4, we prove Theorem \ref{thm:1} with the main highlight showing that a sectorial domain contained in a component of the quasi-Fatou set yields a bound on the growth of the function. Finally in section 5, we prove Theorem \ref{thm:2}, along the way showing the certain sectorial domains in $\R^3$ are ambient quasiballs.

{\bf Acknowledgements:} The author would like to thank Jun Wang for introducing the author to the topic of Julia limiting directions in complex dynamics, Rod Halburd for organizing the {\it CAvid} online seminar series that provided the venue for this introduction and Dan Nicks for providing some helpful comments. 

\section{Preliminaries}

\subsection{Quasiregular mappings}

We refer to Rickman's monograph \cite{Rickman} for a complete exposition on quasiregular mappings. Below we just recall the basics.

A continuous map $f: E \to \R^n$ defined on a domain $E\subset \R^n$ is called \emph{quasiregular}  if $f$ belongs to the Sobolev space $W^{1,n}_{\text{loc}}(E)$ and if there exists some $K\geq 1$ such that 
\[ |f'(x)|^n \leq K J_f(x), \qquad \text{for a.e. $x\in E$}.\]
Here $J_f$ denotes the Jacobian of $f$ at $x\in E$ and $|f'(x)|$ the operator norm. The smallest such $K$ for which this inequality holds is called the {\it outer dilatation} and denoted $K_O(f)$. If $f$ is quasiregular, then we also have
\[ J_f(x) \leq K' \min _{|h|=1}  |f'(x)(h)|, \qquad \text{ for a.e. $x\in E$}.\]
The smallest $K'$ for which this inequality holds is called the {\it inner dilatation} and denoted $K_I(f)$. Then the {\it maximal dilatation} of a quasiregular map $f$ is $K(f) = \max \{K_O(f) , K_I(f) \}$. If $K(f) \leq K$, we then say that $f$ is $K$-quasiregular. 

An injective quasiregular mapping is called {\it quasiconformal}. A {\it quasiball} $B_1$ in $\R^n$ is the image of the unit ball $\B^n$ under a quasiconformal map $f:\B^n \to B_1$ and we call an {\it ambient quasiball} $B_2$ the image of the unit ball $\B^n$ under a quasiconformal map $f:\R^n \to \R^n$, that is, $f(\B^n) = B_2$. We note that there does not appear to be consistency in the literature in how the term quasiball is defined: for example, in dimension $2$ a quasidisk always refers to an ambient quasidisk in our terminology, but in higher dimensions, a quasiball usually (but not always!) does not mean an ambient quasiball.

A non-constant entire quasiregular map $f:\R^n\to \R^n$ is said to be of {\it polynomial-type} if
\[ \lim_{x\to \infty} f(x) = \infty,\]
otherwise $f$ is said to be of {\it transcendental-type}. The order of growth of an entire quasiregular mapping in $\R^n$ is given by \cite[p.121]{Rickman}:
\[ \mu_f = \limsup_{r\to \infty} (n-1) \frac{\log\log M(r,f)}{\log r}.\]

If $f:\R^n \to \R^n$ is quasiregular, then the {\it Julia set} $J(f)$ of $f$ is defined to be the set of all $x\in \R^n$ such that
\[ \operatorname{cap} \left ( \R^n \setminus \bigcup_{k=1}^{\infty} f^k(U) \right ) = 0\]
for every neighbourhood $U$ of $x$. Here, $\operatorname{cap}$ denotes the conformal capacity of a condenser: we refer to \cite[p.53]{Rickman} for the definition. Essentially what this means is that the forward orbit of any neighbourhood of a point in the Julia set can only omit a small set.

The {\it quasi-Fatou set} $QF(f)$ is then the complement of $J(f)$. We remark that one key difference with complex dynamics is that there is no assumed notion of normality of the family of iterates on components of the quasi-Fatou set.

\subsection{Zorich maps}
\label{sect:z}

Zorich maps are higher dimensional analogues of the exponential function $e^z$. We will fix a particular Zorich map $Z:\R^3 \to \R^3$, following the presentation in \cite[Section 5]{NS18}. With $M(x_1,x_2) = \max \{ |x_1| , |x_2| \}$, define a bi-Lipschitz map $h$ from the square $[-1,1]^2$ to the upper faces of the square based pyramid with vertex $(0,0,1)$ and base $[-1,1]^2 \times \{ 0 \}$ by setting
\[ h(x_1,x_2) = (x_1,x_2, 1-M(x_1,x_2) ).\]
See \ref{fig:1}.

\begin{figure}[h]
\begin{center}
\includegraphics{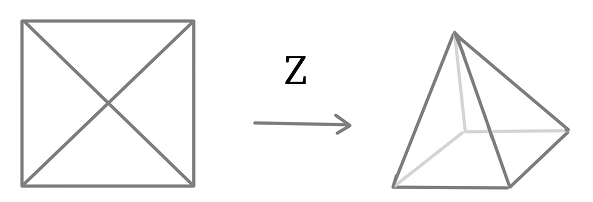}
\end{center}
\caption{The image of the square $[-1,1]^2$ under $Z$.}
\label{fig:1}
\end{figure}

Then in the infinite square cylinder $[-1,1]^2 \times \R$, we define $Z$ via
\[ Z(x_1,x_2,x_3) = e^{x_3}h(x_1,x_2),\]
with image the upper half-space $\{(x_1,x_2,x_3) : x_3 \geq 0 \} \setminus  \{ {\bf 0} \}$, writing ${\bf 0}$ for the origin in $\R^3$. This is then extended via repeated reflections in the faces of the square cylinder in the domain, and in the plane $\{ (x_1,x_2,x_3) : x_3=0 \}$ in the range to yield a quasiregular map $Z:\R^3 \to \R^3 \setminus \{ {\bf 0} \}$ of transcendental-type.

By construction, the image of $[-1,3]\times[-1,1] \times \R$ under $Z$ is all of $\R^3 \setminus \{ {\bf 0} \}$. By removing some of the boundary, we find a set $E$ whose closure is $[-1,3] \times [-1,1] \times \R$ and a branch of the inverse $Z^{-1} : \R^3 \setminus \{ {\bf 0} \} \to E$. We refer to \cite[Section 3.2]{FP} for more details on constructing such branches of the inverse of a Zorich map.

The Zorich map $Z$ is strongly automorphic (that is, periodic and transitive on fibres) with respect to a discrete group of isometries $G$. The group $G$ is generated by translations $x \mapsto x + (4,0,0)$, $x\mapsto x+(0,4,0)$ and a rotation by $\pi$ fixing the line $\{ (1,1,t) : t\in \R \}$. In particular, we note that any beam of the form $[2n-1,2n+1]\times [2m-1,2m+1] \times \R$ is mapped onto a closed half-space with the origin removed.

\subsection{Hyperbolic-type metrics}

For a proper subdomain $U$ of $\R^n$, recall from \cite[p.103]{Vuorinen} the metric $\mu_U$. This metric is conformally invariant and can be viewed as a substitute for the hyperbolic metric on domains which do not carry it. A key property of $\mu_U$ is how it behaves under quasiregular mappings. If $f:U \to V$ is a non-constant quasiregular mapping with inner dilatation $K_I(f)$, then by \cite[Theorem 10.18]{Vuorinen} we have
\begin{equation}
\label{eq:mu}
\mu_{f(U)} ( f(x) , f(y) ) \leq K_I(f) \mu_U(x,y)
\end{equation}
for all $x,y\in U$. If $G$ is a subdomain of $U$, then by \cite[Remark 8.5]{Vuorinen}, $\mu$ satisfies the subordination principle
\begin{equation}
\label{eq:sub} 
\mu_U(x,y) \leq \mu_G(x,y)
\end{equation}
for all $x,y\in G$.

There are two other hyperbolic-type metrics on a proper subdomain $U$ of $\R^n$ that we will need.
First, the quasihyperbolic metric $k_U$ is a conformal metric with density $w(x) = d(x,\partial U)^{-1}$. This means that
\[ k_U(x,y) = \inf _{\gamma} \int_{\gamma} \frac{|dt|}{d(t,\partial U) },\]
for all $x,y \in U$ and where the infimum is taken over all paths $\gamma$ in $U$ joining $x$ to $y$. Clearly $k$ also obeys a subordination principle: if $G$ is a subdomain of $U$, then
\begin{equation}
\label{eq:subk}
k_U(x,y) \leq k_G(x,y)
\end{equation}
for all $x,y \in G$.
Second, the distance-ratio metric $j_U$ is defined by
\[ j_U(x,y) = \log \left ( 1 + \frac{ |x-y| }{\min \{ d(x,\partial U) , d(y,\partial U) \} } \right ),\]
for $x,y\in U$.

If $U$ is as above, then by \cite[Lemma 8.30 (2)]{Vuorinen}, there exist constants $b_1$ and $b_2$, depending only on $n$, such that
\begin{equation}
\label{eq:muk}
\mu_U(x,y) \leq b_1 k_U(x,y) + b_2
\end{equation}
for all $x,y \in U$. If $U$ is a proper subdomain of $\R^n$, then we denote by $\partial_{\infty}U$ the boundary of $U$ in $\R^n \cup \{ \infty \}$. If $\partial_{\infty}U$ is connected, by \cite[Lemma 8.31]{Vuorinen}, there exists a constant $c_n$, depending only on $n$, such that
\begin{equation}
\label{eq:muj}
\mu_U(x,y) \geq c_n j_U(x,y)
\end{equation}
for all $x,y\in U$. This result will be used in the following context: if $U$ is a full subdomain of $\R^n$, that is, $U$ has no bounded complementary components, then $(\R^n \cup \{ \infty \}) \setminus U$ is connected and hence so is $\partial_{\infty} U$ by, for example, \cite[Theorem 1.6(3)]{GMI}.

\section{Every direction is Julia limiting}

In this section, we will prove Theorem \ref{thm:1}. First, we define the following sectorial regions in $\R^n$, for $n\geq 2$. For $x_0\in \R^n$, $\theta \in S^{n-1}$ and $0< \eta \leq \pi$, let
\[ \Omega(x_0,\theta,\eta) = \{x\in \R^n : d ( \sigma(x-x_0) , \theta ) < \eta \},\]
where we recall that $\sigma(x) \in S^{n-1}$ is the spherical component of $x$ in spherical coordinates, and for $p,q\in S^{n-1}$, $d(p,q)$ denotes the angle between the line segments (in $\R^n$) $[0,p]$ and $[0,q]$.

First, we need an estimate on the quasihyperbolic metric in sectorial regions. We write $e_1 = (1,0,\ldots, 0) \in S^{n-1}$.

\begin{lemma}
\label{lem:ksect}
For $x_0\in \R^n$, $\theta\in S^{n-1}$ and $0<\eta \leq \pi$, fix $x_1$ on the ray $R = \{x\in \R^n : \sigma(x-x_0) = \theta \}$.
Then there exists a constant $C>0$ depending only on $\eta$ such that for $x\in R$ with $|x| > |x_1|$, we have
\[ k_{\Omega(x_0,\theta,\eta )} ( x,x_1) \leq C \log \frac{|x|}{|x_1|} .\]
\end{lemma}

\begin{proof}
By applying a translation and a rotation, and noting the invariance of the quasihyperbolic metric under such maps, we may assume that $x_0=0$ and $\theta = e_1$.

First, if $\eta \geq \pi/2$, then for $x\in R$ we have that $d(x,\partial \Omega)$ is realised by $d(x, {\bf 0})$ and hence $d(x,\partial \Omega) = |x|$. The result then follows with $C=1$, since the path realising $k_{\Omega}(x,x_1)$ will be the line segment joining $x$ and $x_1$ contained in the ray $R$.

If  $\eta <\pi/2$, then by Euclidean geometry we have $d(x,\partial \Omega) = |x| \sin \eta$. The result again follows, this time with $C = 1/\sin \eta$.
\end{proof}

Our main tool is the following result, a higher dimensional version of \cite[Lemma 1]{Qiao}. For a proper subset $U$ of $\R^n$, we write $T(U)$ for the topological hull of $U$, that is, the union of $U$ with all its bounded complementary components.

\begin{theorem}
\label{lem:sector}
Let $f:\R^n \to \R^n$ be a $K$-quasiregular mapping of transcendental-type. Suppose that for some $x_0 \in \R^n$, $\theta \in S^{n-1}$ and $\eta >0$ the sector $\Omega(x_0,\theta,\eta)$ is contained in a component $U$ of the quasi-Fatou set $QF(f)$ for which $T(U)$ is a proper subset of $\R^n$. Then if $0 < \eta ' < \eta$, there exists a constant $d$ depending only on $n,K$ and $\eta-\eta '$ such that
\[ |f(x)| = O( |x|^d)\]
for $x\in \Omega(x_0,\theta, \eta')$.
\end{theorem}

The condition that $T(U)$ is a proper subset of $\R^n$ is redundant in the special case that $f$ is a transcendental entire function in the plane. Baker \cite[Theorem 3.1]{Baker84} showed that every multiply connected Fatou component is bounded. It is still open if the same is true for components of $QF(f)$ in the quasiregular setting.

\begin{proof}
Suppose that $\Omega := \Omega(x_0,\theta, \eta)$ is contained in the component $U$ of the quasi-Fatou set $QF(f)$. Then $f(\Omega)$ belongs to a component $V$ of $QF(f)$. It follows that $T(V)$ is a proper subset of $\R^n$. If not, then $V$ must necessarily be an unbounded and hollow (that is, not full) component of $QF(f)$. It follows by \cite[Theorem 1.4]{NS17} that $V$ is completely invariant and hence $T(U)$ is all of $\R^n$, contrary to the hypothesis.
Hence, by the discussion after equation \eqref{eq:muj}, $\partial_{\infty}T(V)$ is a connected subset of $\R^n \cup \{\infty\}$ and we may apply \eqref{eq:muj} to $T(V)$.

By the subordination principle \eqref{eq:sub}, \eqref{eq:mu}, \eqref{eq:sub} again and \eqref{eq:muj}, we have for any $x,y\in \Omega$ that
\begin{align*}
K_I(f) \mu_{\Omega}(x,y) &\geq K_I(f) \mu_U(x,y) \\
&\geq \mu_V(f(x),f(y)) \\
&\geq \mu_{T(V)} ( f(x),f(y)) \\
& \geq c_n j_{T(V)} ( f(x),f(y) ).
\end{align*}

We now fix $y \in \Omega$. For definiteness and without loss of generality we may take $y$ to be on the ray $R = \{ z \in \R^n : \sigma(z-x_0) = \theta \}$ with $|y-x_0|=1$. Then by the above chain of inequalities, \eqref{eq:muk} and the definition of $j_{T(V)}$, for $x\in \Omega$ we have
\begin{equation}
\label{eq:1}
c_n \log \left ( 1 + \frac{ |f(x) - f(y)|}{ \min \{ d(f(x) , \partial T(V) ) , d(f(y) , \partial T(V)) \} } \right ) \leq K_I(f) \left ( b_1 k_{\Omega}(x,y) + b_2 \right ).
\end{equation}

We now restrict $x$ to be in the sector $\Omega' = \Omega(x_0,\theta, \eta')$. Given such an $x$, let $R_x$ be the ray emanating from $x_0$ and passing through $x$. If $y' $ is the unique point in $R_x$ such that $|y'| = |y|$, then for $|x|$ large enough, we have
\begin{equation} 
\label{eq:2}
k_{\Omega}(x,y) \leq k_{\Omega}(x,y') + k_{\Omega}(y',y) \leq 2 k_{\Omega}(x,y').
\end{equation}
Setting $\Omega_x$ to be the maximal sector of the form $\Omega(x_0,\sigma(x-x_0) , t)$ contained in $\Omega$, we apply Lemma \ref{lem:ksect} to $\Omega_x$ to see that there is a constant $C$ depending on $\eta - \eta'$ such that
\begin{equation}
\label{eq:3}
k_{\Omega_x}(x,y') \leq C \log \frac{|x|}{|y'|}.
\end{equation}

Combining \eqref{eq:subk}, \eqref{eq:1}, \eqref{eq:2} and \eqref{eq:3}, there are constants $A,B$ depending only on 
$K_I(f)$, $n$ and $\eta - \eta'$ such that
\begin{equation}
\label{eq:4}
\log \left ( 1 + \frac{ |f(x) - f(y)|}{ \min \{ d(f(x) , \partial T(V) ) , d(f(y) , \partial T(V)) \} } \right )
\leq A \log \frac{ |x|}{|y| } + B.
\end{equation}

Using the facts that $\log$ is increasing, that $y$ is fixed and that  
\[ \min \{ d(f(x) , \partial T(V)) , d(f(y) , \partial T(V) ) \} \leq d(f(y) , \partial T(V) ),\]
\eqref{eq:4} yields constants $A', B'$, again depending only on $K_I(f)$, $n$ and $\eta - \eta'$ such that
\[ |f(x)| \leq A' |x|^{B'}\]
for all $|x| \in \Omega '$ with sufficiently large modulus. The result now follows.
\end{proof}

We can now prove Theorem \ref{thm:1}.

\begin{proof}[Proof of Theorem \ref{thm:1}]
Suppose that $L(f)$ is not all of $S^{n-1}$. Then since $L(f)$ is closed, the complement in $S^{n-1}$ is open. Hence if $\xi \in S^{n-1} \setminus L(f)$, there exist $\eta >0$ and $x_0 \in \R^n$ such that $\Omega(x_0,\xi , \eta) \cap J(f) = \emptyset$. Therefore, by Theorem \ref{lem:sector}, for $0<\eta' < \eta$, we have
\[ |f(x)| = O(|x|^d)\]
for $x\in \Omega(x_0,\xi , \eta ')$. In particular, we can choose $R_1 >0$ large enough so that there exists $C_1>0$ with
\begin{equation}
\label{eq:5}
|f(x) | \leq C_1 |x|^d
\end{equation}
for $x\in \Omega(x_0,\xi , \eta ')$ and $|x| >R_1$. Now, by \cite[Lemma 3.4]{Berg05}, since $f$ is of transcendental-type, we have
\[ \lim_{r\to \infty} \frac{\log M(r,f) }{\log r} = \infty.\]
We may choose $R_2>0$ large enough that
\[ \log M(r,f) > d\log r + \log \left ( \frac{2C_1 \alpha^d}{\delta} \right ),\]
for $r>R_2$, recalling $\alpha$ and $\delta$ from the minimum modulus hypothesis. From this and the hypotheses, for $r > \max \{R_1,R_2\}$ and some $s\in[r,\alpha r]$, we have
\begin{align*}
m(s,f) & \geq \delta M(r,f) \\
&\geq \delta \left ( \frac{2C_1 \alpha^d}{\delta} \right ) r^d \\
&= 2C_1 \alpha^d \left ( \frac{r}{s} \right )^d s^d \\
&\geq 2C_1 s^d.
\end{align*}
This contradicts \eqref{eq:5} and proves the theorem.
\end{proof}

\section{The inverse problem}

In this section, we will prove Theorem \ref{thm:2}. 

\subsection{The construction of Nicks and Sixsmith}

The main tool here will be a construction from \cite[Theorem 5]{NS18} of a quasiregular map $F:\R^3 \to \R^3$ of transcendental-type which is a higher dimensional version of $e^z +z$. A key property of this map is that it is the identity in half-space $H = \{ (x_1,x_2,x_3) \in \R^3 : x_3<0 \}$.
We also note that since $F$ agrees with $Z+Id$ in a half-space $H' = \{ (x_1,x_2,x_3) : x_3 > T\}$ for some $T>0$, where $Z$ is the Zorich map from Section \ref{sect:z}, and is the identity in another half-space with bi-Lipschitz interpolations between, it follows that the order is
\begin{equation}
\label{eq:order}
\mu_F = \limsup_{r\to \infty} (3-1) \frac{\log \log M(r,f)}{\log r} = \limsup_{r\to \infty} 2 \frac{ \log \log (e^r+r)}{\log r} = 2.
\end{equation}
Moreover, the lower order (where the $\limsup$ is replaced by $\liminf$) is also $2$.

In \cite{NS18}, the authors did not need any information on the Julia set. We, however, do need to identify the rough location of certain points in the Julia set to be able to determine the set of Julia limiting points. First, we need a covering result.

\begin{lemma}
\label{lem:covering}
There exist constants $\alpha >T, \beta>0$ such that if
\[ B(x,\beta) \subset U_{\alpha} := \{ (x_1,x_2,x_3)\in \R^3 : x_3 > \max ( \alpha, (x_1^2+x_2^2)^{1/4} ) \}, \] 
then there is a continuum in $F(B(x,\beta))$ that separates ${\bf 0}$ from $\infty$.
\end{lemma}

Here, a bounded continuum $Y$ separates ${\bf 0}$ from $\infty$ if ${\bf 0}$ is in $\R^n \setminus Y$ and not in the unbounded component of $\R^n \setminus Y$. We remark that this lemma is clearly true for $Z$ itself, since $Z$ maps $\{x_3 = s \}$ onto a topological sphere which separates ${\bf 0}$ from $\infty$, but more work has to be done for $F$.

\begin{proof}
Suppose that $x \in U_{\alpha}$. 
Writing $x = (x_1,x_2,x_3)$, via the triangle inequality we have
\begin{equation}
\label{eq:a} 
|Z(x)| - |x| \leq |F(x)| \leq |Z(x)| + |x|,
\end{equation}
and since $x_3 > (x_1^2 + x_2^2)^{1/4}$ it follows that $|x| < (x_3^4 + x_3^2)^{1/2}<2x_3^2$ if $\alpha$ is sufficiently large.
Hence given any $\epsilon >0$, we can choose $\alpha$ large enough that if $x_3 > \alpha$ then $e^{-x_3}|x| < \epsilon$. Moreover, by the construction of $Z$, there exist positive constants $C_1 < C_2$ such that
\[ C_1e^{x_3} \leq |Z(x)| \leq C_2 e^{x_3}\]
for all $x\in\R^n$. It then follows from \eqref{eq:a} that for $x\in U_{\alpha}$ we have
\begin{equation}
\label{eq:b} 
\frac{C_1 e^{x_3}}{2} \leq |F(x)| \leq 2C_2 e^{x_3}.
\end{equation}

For $m,n\in \Z$, let $\Omega_{m,n}$ be the square $[2m-1,2m+1] \times [2n-1,2n+1]$. Since $Z$ is defined via a PL map on each slice $\Omega_{m,n} \times \{s \}$ for (fixed) $m,n \in \Z$ and $s\in \R$ and $F(x) = Z(x) + x$, then $F$ is also defined via a PL map on each $\Omega_{m,n} \times \{s \}$. 
Since $Z$ maps the edges of the square $\Omega_{m,n} \times \{ s \}$ into the plane $\{x_3 = 0\}$, it follows that $F$ maps these edges into the plane $\{ x_3 = s\}$. 

Now suppose $m,n \in \Z$ and $s\in \R$ are such that three neighbouring squares have the property that
\[ \bigcup_{i=m-1}^{m+1} \Omega_{i,n} \times \{ s \} \subset U_{\alpha}.\]
Denote by $P_{m-1},P_m$ and $P_{m+1}$ the respective images of these squares under $F$, and note that each of these is a pyramid. 
Write $w_1,w_2,w_3,w_4$ for the vertices of $\Omega_{m,n} \times \{ s\}$ starting at the top left and moving in an anticlockwise direction. Then $w_1,w_2$ are vertices shared with $\Omega_{m-1,n} \times \{ s\}$ and $w_3,w_4$ are vertices shared with $\Omega_{m+1,n} \times \{ s \}$.

Now, the vertices of $\Omega_{m,n} \times \{s \}$ and $\Omega_{m-1,n} \times \{ s\}$ that are not common to both are related by a translation $\tau (x) = x + (4,0,0)$ in the group $G$ with respect to which $Z$ is strongly automorphic. Since
\[ F(\tau(x)) = Z(\tau (x)) + \tau (x) = Z(x) + x + (4,0,0) = F(x) + (4,0,0),\]
it follows that the vertices of $P_m,P_{m-1}$ in $\{x_3 = s \}$ which are not common must differ by the vector $(4,0,0)$. The same property is true for $P_m$ and $P_{m+1}$. We conclude that $P_m$ and $P_{m-1}$ are pyramids whose bases are contained in the plane $\{x_3 = s\}$ and don't quite match up, see Figure \ref{fig:2}

\begin{figure}[h]
\begin{center}
\includegraphics{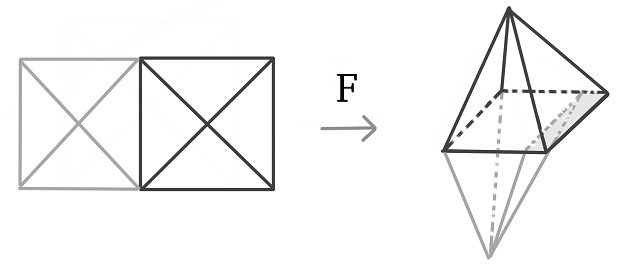}
\end{center}
\caption{The darker shade indicates $\Omega_{m,n} \times \{ s \}$ and its image under $F$, and the ligher shade indicates $\Omega_{m-1,n} \times \{ s \}$ and its image under $F$. The gap between the pyramids in the plane $\{x_3 = s\}$ is shaded.}
\label{fig:2}
\end{figure}

By the same argument above, since $\Omega_{m+1,n}\times \{ s \} = \tau ( \Omega_{m-1,n} \times \{s \} )$, it follows that $P_{m+1}$ must be a translate of $P_{m-1}$ by $(4,0,0)$. Consequently, this and \eqref{eq:b} imply that the union 
\[ P = \bigcup_{i=m-1}^{m+1} P_i\]
is contained in the ring $\{x : C_1e^s/2 < |x| < 2C_2e^s \}$ and separates $0$ from $\infty$. To see why this is true, consider the slice of $P$ intersected with the plane $\{x_2  = w \}$, see Figure \ref{fig:3} for a typical instance of this, and recall $P_{m-1}$ and $P_{m+1}$ are translates.

\begin{figure}[h]
\begin{center}
\includegraphics{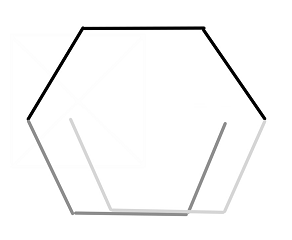}
\end{center}
\caption{A slice of $P$ intersected with $\{x_2 = w \}$. The darker shade indicates $P_m$, the lighter shade $P_{m-1}$ and the dotted line indicates $P_{m+1}$.}
\label{fig:3}
\end{figure}

Finally, we choose $\beta$ large enough so that $B(x,\beta)$ is guaranteed to contain three consecutive squares as above. Then $F(B(x,\beta))$ must contain a continuum separating $0$ from $\infty$.
\end{proof}

We can now identify the Julia limiting directions of $F$.

\begin{lemma}
\label{lem:Fjld}
With $F:\R^3 \to \R^3$ as above, the set of Julia limiting directions $L(F)$ is the hemisphere $S^2_+ := \{ (x_1,x_2,x_3) \in S^2 : x_3 \geq 0 \}$.
\end{lemma}

\begin{proof}
Since the lower order of $F$ is positive, it follows by \cite[Theorem 1.2]{BFN} that the Julia set $J(F)$ agrees with the boundary of the fast escaping set $\partial A(F)$. Here, we do not need the precise definition of the fast escaping set, only that it is non-empty for mappings of transcendental-type and that every component is unbounded, see \cite[Theorem 1.2]{BDF}.

Let $X$ be a component of $A(F)$. Then there exists $R>0$ such that $X$ meets $S(0,r)$ for every $r>R$. Since the half-space $H$ on which $F$ is the identity is clearly not in $A(F)$, it follows that $\partial X$, and hence $J(F)$, meets $S(0,r)$ for every $r>R$. More generally, if $Y$ is any continuum contained in $\R^3 \setminus B(0,R)$ that separates $0$ from $\infty$, then $Y$ is guaranteed to meet $\partial X$ and hence $J(F)$.

Choose $\alpha$ large enough that if $x\in U_{\alpha}$, then $F(B(x,\beta)) \subset \R^3\setminus B(0,R)$. By Lemma \ref{lem:covering} and the previous observation, it follows that $F(B(x,\beta))$ must meet $J(F)$. By complete invariance of the Julia set, it follows that $J(F)$ meets $B(x,\beta)$ too.

Now, if $\theta \in S^2_+$, choose $x_n$ with $B(x_n,\beta) \subset U_{\alpha}$ such that $|x_n| \to \infty$ and $\sigma(x_n) \to \theta$ as $n\to \infty$. It's important to observe that if $\xi =(a,b,0) \in S^2$, then we can find such a sequence since $U_{\alpha}$ is not contained in any sector of the form $\Omega(0,(0,0,1), \eta)$ for $\eta < \pi /2$.

Finally,  find $y_n \in J(F) \cap B(x_n,\beta)$. By construction, we have $|y_n|\to \infty$ and $\sigma(y_n) \to \theta$ and hence $\theta \in L(F)$. It follows that $S^2_+$ is contained in $L(F)$. To see that $L(F)$ equals $S^2_+$, suppose that $\xi \in L(F) \cap (S^2 \setminus S^2_+)$. Then there must exist a sequence $(y_n)_{n=1}^{\infty}$ in $J(F)$ with $\sigma(y_n) \to \xi$, but also can be assumed to be in the half-plane $H = \{ (x_1,x_2,x_3) : x_3<0 \}$ on which $F$ is the identity. This yields a contradiction, and the lemma follows.
\end{proof}

\subsection{Double inner chordarc curves and sectorial domains}

We now prove the following lemma that shows certain sectorial domains are ambient quasiballs, recalling that an ambient quasiball is the image of a ball or half-space under an ambient quasiconformal map.

\begin{lemma}
\label{lem:qbsect}
Let $J\subset S^2$ be a double inner chordarc domain, and let $D$ be the closure of one of the components of $S^2 \setminus J$. Writing $x\in \R^3 \setminus \{  0 \}$ in spherical coordinates $x=(r,\theta) \in \R^+ \times S^2$, let $S$ be the sectorial domain
\[ S = \{ (r,\theta) : r>0, \theta \in D \} . \]
Then there is a quasiconformal map $f_S:\R^3 \to \R^3$ such that $f_S(S) =H$.
\end{lemma}

\begin{proof}
We fix a branch $Z^{-1}$ of the inverse of the Zorich map which maps $\R^3 \setminus \{ 0 \}$ onto a beam $B\subset [-1,3]\times [-1,1]  \times \R$ in $\R^3$ which includes some, but not all, of the boundary of $[-1,3] \times [-1,1] \times  \R$, recall Section \ref{sect:z}.

Next, find $\theta_0 \in S^2$ and $0<\delta < \pi$ so that $S \subset \Omega(0,\theta_0,\delta)$. By applying a rotation, we can assume without loss of generality that $\theta_0 = e_1$. 
Let $h:S^2 \to S^2$ be a lift of the planar linear map $z\mapsto \epsilon z$ for $\epsilon >0$ such that $h$ fixes $e_1$ and is the attracting fixed point. Then $h$ is bi-Lipschitz in the chordal metric and we can extend $h$ radially to obtain a quasiconformal map
$f_1 :\R^3 \to \R^3$ such that $f_1(S) \subset \{ x=(x_1,x_2,x_3) : x_3 >0 \}$. Then since $f_1$ is also bi-Lipschitz and $Z^{-1}$ restricted to any slice is bi-Lipschitz and since the internal chord-arc condition is preserved under bi-Lipschitz mappings, it follows that $Z^{-1}(f_1(S))$ is a cylinder contained in $[-1,1]^2 \times\R$ whose cross-section is a domain with the internal chordarc condition. By \cite[Theorem 5.2]{Vaisala} it follows that there is a quasiconformal map $f_2:Z^{-1}(f(S)) \to \B^3$.

We conclude that there is a quasiconformal map $g_S:S \to \B^3$. Applying the same argument to the sectorial domain $S^* := \R^3 \setminus \overline{S}$, we conclude there is also a quasiconformal map $g_{S^*}:S^* \to \B^3$. Hence, by \cite{Gehring}, it follows that $S$ is an ambient quasiball, and the conclusion follows.
\end{proof}

To illustrate the main idea behind the proof of Theorem \ref{thm:2}, we prove the following special case.

\begin{lemma}
\label{lem:m=1}
Let $J\subset S^2$ be a double inner chordarc domain, and let $D$ be the closure of one of the components of $S^2 \setminus J$. Let $S$ be the sector $\R^+ \times D$ in spherical coordinates. Then there exists a quasiregular mapping $F_D$ of transcendental-type of finite order such that $L(F_D) = D$. Moreover, $F_D$ is the identity in $\R^3 \setminus S$.
\end{lemma}

\begin{proof}
With $F$ the construction of Nicks and Sixsmith and $f_S$ the quasiconformal map from Lemma \ref{lem:qbsect}, define 
\[ F_D = f_S^{-1} \circ F \circ f_S.\]
Clearly $F_D$ is a quasiregular mapping of transcendental-type, and the conjugate of a map of finite order by a quasiconformal map is also of finite order (although the order may change).
The Julia set of $F_D$ is $f_S^{-1}(J(F))$. By Lemma \ref{lem:Fjld} it follows that $L(F_D) = D$.
\end{proof}

\begin{proof}[Proof of Theorem \ref{thm:2}]
Let $D_1,\ldots, D_m$ and $E = \bigcup_{j=1}^m D_j$ be as in the hypotheses. Construct the quasiregular maps $F_{D_1} , \ldots, F_{D_m}$ via Lemma \ref{lem:m=1}. The required map is then
\[ f_E(x) = \left \{ \begin{array}{cc} F_{D_j}(x) & \sigma(x)\in D_j, \: j=1,\ldots, m \\ x & \text{otherwise} \end{array} \right .  .\]
 Since the sectors $S_1,\ldots, S_m$ are pairwise disjoint, $f_E$ is a quasiregular map of transcendental-type and is of finite order.

Since $f_E$ is the identity outside $\bigcup_{j=1}^m S_j$, Lemma \ref{lem:qbsect} shows that $L(f_E) = E$. It is worth pointing out that $J(f_E)$ is not necessarily the union of $J(F_{D_1}) , \ldots, J(F_{D_m})$ since, for example, there could be elements of $S_1 \cap J(F_{D_1})$ which are mapped by $F_{D_1}$ into $J(F_{D_2})$. However, these extra points in the Julia set must remain in the union of the sectors, and so do not affect the result.
\end{proof}

\end{document}